\documentclass[reqno,11pt]{amsart}

\usepackage{graphicx}
\usepackage{url}
\usepackage[cmtip,arrow]{xy}  
\usepackage{pb-diagram,pb-xy}  
\usepackage{graphicx}    
\usepackage{amsfonts}
\usepackage{amsmath, amssymb}
\usepackage[latin1]{inputenc}
\usepackage{color, enumerate}

\usepackage{pgfplots}
\usepackage{subfigure}
\usepackage{caption}
\usepackage{float}

\usepackage{amsmath}
\usepackage{amsthm}
\usepackage{amsfonts}
\usepackage{array}
\newcolumntype{C}[1]{>{\centering\arraybackslash}p{#1}}

\newcommand{\SF}{\varepsilon}

\newcommand{\T}{\mathbb{T}}
\newcommand{\F}{\mathcal{F}}

\renewcommand{\SF}{{\mathcal{F}}}

\newcommand{\SL}{{\mathcal{L}}}

\newcommand{\ST}{{\mathcal{T}}}
\newcommand{\SU}{{\mathcal{U}}}

\newcommand{\Z}{\mathbb{Z}}

\newcommand{\R}{\mathbb{R}}

\newcommand{\id}{\textup{id}}

\renewcommand{\mod}{\textup{mod}}

\newcommand{\we}{\wedge}
\newcommand{\pp}[2]{\frac{\partial#1}{\partial#2}}

\newcommand{\bd}{\partial}

\newcommand{\frbd}[1]{\frac{\bd}{\bd #1}}

\newtheorem{proposition}{Proposition}
\newtheorem{theorem}[proposition]{Theorem}
\newtheorem{definition}[proposition]{Definition}

\theoremstyle{remark}
\newtheorem{remark}[proposition]{Remark}

\begin{document}

\bibliographystyle{alpha}

\title[Cotangent models for integrable systems]{Cotangent models for integrable systems}

\author{Anna Kiesenhofer}

\author{Eva Miranda}
\address{Department of Mathematics
 Universitat Polit\`{e}cnica de Catalunya and BGSMath, EPSEB, Avinguda del Doctor Mara\~{n}\'{o}n 44--50, Barcelona, Spain}
\email{anna.kiesenhofer@upc.edu}\email{eva.miranda@upc.edu}
\thanks{ Both authors are supported by the grants reference number MTM2015-69135-P (MINECO/FEDER) and reference number 2014SGR634 (AGAUR). A. Kiesenhofer is supported by a UPC doctoral grant.}

\date{\today}
\begin{abstract} We associate \emph{cotangent models}  to a neighbourhood of a Liouville torus in symplectic and Poisson manifolds focusing on $b$-Poisson/$b$-symplectic manifolds.  The semilocal equivalence with such models uses the corresponding action-angle  theorems in these settings:  the theorem of Liouville-Mineur-Arnold \cite{arnold} for symplectic manifolds and an action-angle theorem for regular Liouville tori in Poisson manifolds \cite{Laurent-Gengoux2010}. Our models comprise regular Liouville tori of Poisson manifolds but also consider the Liouville tori on the singular locus of a $b$-Poisson manifold. For this latter class of Poisson structures  we define a \emph{twisted cotangent model}. The equivalence with this twisted cotangent model is given by  an action-angle theorem recently proved in  \cite{KMS}.  This viewpoint of cotangent models provides a new machinery to construct examples of integrable systems, which are especially valuable in  the $b$-symplectic case where  not many sources of examples are known.   At the end of the paper we introduce non-degenerate singularities as lifted cotangent models on $b$-symplectic manifolds and discuss some generalizations of these models to general Poisson manifolds.
\end{abstract}
\maketitle


\section{Introduction}

The action-angle theorem of Liouville-Mineur-Arnold (\cite{arnold}, \cite{duistermaat}) gives a symplectic description of an integrable system in a neighbourhood of a Liouville torus\footnote{In \cite{duistermaat} obstructions to the global validity of this model are determined.} as a trivial fibration by Lagrangian tori (Liouville tori). In this paper we present a reformulation of the Liouville-Mineur-Arnold theorem as a symplectic equivalence in a neighbourhood of a Liouville torus to an integrable system determined by the cotangent lift of the actions by translations on the Liouville torus. Having such a cotangent lift model for integrable systems is useful to produce examples as lifts of abelian actions on the base. The  Hamiltonian nature of the lifted action is automatic  (see \cite{guilleminandsternberg}) and the fact that the action on the base is given by an abelian group  yields an integrable system on the total space.

  This paper is also devoted to establishing cotangent models for integrable systems on Poisson manifolds with a special focus on  $b$-Poisson manifolds.

   In this paper not only regular Liouville tori on Poisson manifolds are considered but also Liouville tori lying on the singular locus of a  class of Poisson manifolds called $b$-Poisson manifolds. $b$-Symplectic/$b$-Poisson manifolds have been the object of study of recent works in Poisson geometry (cf. \cite{Guillemin2011}, \cite{Guillemin2012}, \cite{Guillemin2013}, \cite{gualtierietal}, \cite{KMS}  and references therein). For these manifolds, such a reformulation is possible via a  dual Liouville form obtaining what  we call the \emph{twisted $b$-cotangent lift} in Section \ref{sec:general_cotlift}. The equivalence with these models uses the corresponding theorem of action-angle coordinates \cite{KMS}.  This new point of view turns out to be very fruitful because it provides a handful of examples of $b$-integrable systems, which was missing in the literature and which we construct in Section \ref{sec:examplesb}.  One of the families of examples is produced by considering abelian symmetries of affine manifolds and applying the twisted $b$-cotangent lift recipe to generate  $b$-integrable systems.

In Section \ref{sec:general_cotlift} we also consider the canonical $b$-cotangent lift which can be used to furnish (singular) examples of Hamiltonian actions on $b$-symplectic manifolds.

Singularities of integrable systems are present in mechanical systems and they correspond to equilibria of  Hamiltonian systems. From a topological point of view, an integrable system on a compact manifold must have singularities. In \cite{eliasson1}, \cite{miranda}, \cite{miranda1}, \cite{mirandazung}, in total analogy with the Liouville-Mineur-Arnold theorem, a symplectic Morse-Bott theory is constructed in a neighbourhood of a singular compact orbit.

In Section \ref{sec:examples}  we present non-degenerate singular integrable systems in the  $b$-symplectic case as  twisted $b$-cotangent lifts of actions by abelian groups which have fixed points on the base or are non-compact. This provides several examples with different kinds of singularities (elliptic, hyperbolic, focus-focus).  This section is an invitation to the study of singularities of integrable systems on $b$-symplectic manifolds. We plan to study normal form theorems for these singularities in $b$-symplectic manifolds  as equivalence to the twisted $b$-cotangent models in the future, thus readdressing the normal form theory already initiated in \cite{Guillemin2012}. We end up the paper discussing the models for the general Poisson case. The action-angle coordinate theorem proved in \cite{Laurent-Gengoux2010} is used to give a product-type model in a neighbourhood of a regular Liouville torus of cotangent lift with a parameter space endowed with the trivial Poisson structure. A starting point of the study of general models for Liouville tori on non-regular Poisson leaves is the action-angle theorem for non-commutative systems proved in \cite{Laurent-Gengoux2010}.

\section{The set-up}\label{sec:pre}

In this section we review some basic results concerning integrable systems on symplectic and Poisson manifolds. In the wide class of Poisson manifolds we single out a subclass called $b$-Poisson manifolds or $b$-symplectic manifolds. The proofs of the main results in this section are contained in other papers in the literature to which we refer. We hint at the main ideas of the proofs.

\subsection{Concerned manifolds}

A \textbf{symplectic manifold} is an even dimensional manifold endowed with a closed non-degenerate $2$-form $\omega$. The non-degeneracy of $\omega$ entails orientability of the underlying manifold and yields a duality between forms and vector fields generalizing the correspondence between functions and Hamiltonian vector fields.

 A \textbf{Poisson manifold} is a pair $(M, \Pi)$ where $\Pi$ is a bivector field satisfying the integrability equation $[\Pi, \Pi]=0$ where $[\cdot,  \cdot]$ stands for the Schouten bracket. In contrast to symplectic manifolds,  there are no constraints on the dimension or orientability of a Poisson manifold.  A Hamiltonian vector field is simply defined via the equality
$X_f= \Pi(df,\cdot)$ and the Poisson bracket can be expressed as  $\{f_i,f_j\}:=\Pi(df_i,df_j)$ .  The equations associated to many mechanical systems are often better formulated in the Poisson language (such as the Gelfand-Ceitlin systems); others can only be formulated in the Poisson context because the system has parameters that blow up.

A class of Poisson manifolds called $b$-Poisson manifolds was recently introduced and studied in
  \cite{Guillemin2011} and \cite{Guillemin2012}. These manifolds are symplectic away from a hypersurface $Z$; along $Z$ the symplectic form  has a certain controlled singularity. We first want to describe this singularity from the Poisson viewpoint:
 A symplectic structure $\omega$, which is a section of $\bigwedge^2 T^\ast M$ induces a ``dual" bivector field $\Pi$, i.e. a section of $\bigwedge^2 T M$:
$$\Pi(df,dg):=\omega(X_f,X_g)=\{f,g\},\qquad f,g\in C^\infty(M).$$
 It can be shown that the bivector field $\Pi$ associated to a symplectic form satisfies the Jacobi identity, which means that it is a {\it Poisson} bivector field. Now consider the case where we start with a symplectic form on $M\backslash Z$ whose dual Poisson structure vanishes along $Z$ in the following controlled way. The definition below is due to {Guillemin-Miranda-Pires (\cite{Guillemin2011}, \cite{Guillemin2012}}),
\begin{definition}\label{definition:firstb} Let $(M^{2n},\Pi)$ be an oriented Poisson manifold. If the map
$$p\in M\mapsto(\Pi(p))^n\in\bigwedge^{2n}(TM)$$
is transverse to the zero section, then $\Pi$ is called a \textbf{$b$-Poisson structure} on $M$. The hypersurface $Z=\{p\in M|(\Pi(p))^n=0\}$ is the {\bf critical hypersurface} of $\Pi$. The pair $(M,\Pi)$ is called a \textbf{$b$-Poisson manifold}.
\end{definition}

This transversality condition gives powerful information about the $b$-Poisson structure. In particular, Weinstein's splitting theorem \cite{weinstein} looks particularly simple for $b$-Poisson structures and a Darboux-type result can be obtained for them as it was proved in \cite{Guillemin2012}.

\begin{theorem}[{\bf $b$-Darboux theorem, \cite{Guillemin2012}}]\label{thm:bdarboux} Let $(M, \Pi)$ be a $b$-Poisson manifold. Then, on a neighbourhood of a point $p\in Z$ in the critical surface, there exist coordinates $(x_1,y_1,\dots ,x_{n-1},y_{n-1}, z, t)$  centered at $p$ such that the critical hypersurface is given by $z=0$ and
$$\omega=\sum_{i=1}^{n-1} \frac{\partial}{\partial{x_i}}\wedge \frac{\partial}{\partial{y_i}}+{z}\,\frac{\partial}{\partial{z}}\wedge \frac{\partial}{\partial{t}}.$$
\end{theorem}

This theorem can be proved using a  \emph{hands-on proof} following the constructive steps of Weinstein's splitting theorem and applying an ad-hoc transformation at the last step using the transversality assumption.

For a reason that will become clear below such Poisson manifolds can be treated dually as forms with poles and are often referred to as $b$-symplectic manifolds. One of the benefits of dealing with such Poisson structures using forms is that, for example, a proof of the $b$-Darboux theorem can be given using the Moser's path method for these generalized forms as it was done in \cite{Guillemin2012}.

\subsection{A crash course on $b$-Poisson manifolds}

In this section we present basic results on $b$-Poisson manifolds that will be needed in this paper. Most of the results in the study of the geometry of $b$-Poisson manifolds are contained in \cite{Guillemin2011}and \cite{Guillemin2012} to which we defer for their proofs. The proofs of the results concerning group actions on these manifolds can be found in \cite{Guillemin2013}.

\subsubsection{Poisson geometry of the critical hypersurface}

One of the immediate consequences of the definition of $b$-Poisson manifolds is that the critical hypersurface is a smooth hypersurface. Not only that, as it was proved in \cite{Guillemin2012} it is also a regular Poisson submanifold of $(M,\Pi)$. In other words, the Poisson structure $\Pi$ induces a Poisson structure on $Z$ which is regular. In \cite{Guillemin2011}, the geometry and topology of the symplectic foliation induced in $Z$ was completely described.  As it was proved in \cite{Guillemin2011}, the symplectic leaves in $Z$ are all symplectomorphic. Moreover if $Z$ is compact, then $Z$ is a mapping torus.
\begin{theorem}[\textbf{Guillemin-Miranda-Pires, \cite{Guillemin2011}}] The critical hypersurface of a $b$-Poisson manifold is a Poisson submanifold with symplectomorphic symplectic leaves. If $Z$ is compact then $Z$ is a mapping torus associated to the flow of a Poisson vector field transverse to the symplectic foliation in $Z$.
\end{theorem}

In view of the result above, {\bf from now on we will assume that $Z$ is  compact}.

Grosso modo, the way this theorem was proved was using a particular vector field transverse to the symplectic foliation preserving the Poisson structure and flowing along it to get the symplectomorphism. The existence of such a vector field is a particular feature of this class of Poisson manifolds.

 For a given volume form $\Omega$ on a Poisson manifold $M$ the associated {\bf modular vector field} $u_{\mod}^\Omega$ is defined as the following derivation:
 \[
C^\infty(M) \to \R : \, f\mapsto \frac{\SL_{X_f}\Omega}{\Omega}.
\]
It can be shown (see for instance \cite{Weinstein2}) that this is indeed a derivation and, moreover, a Poisson vector field. Furthermore, for different choices of volume form $\Omega$, the resulting vector fields only differ by a Hamiltonian vector field.

 In \cite{Guillemin2011} it was shown that $Z$ is the mapping torus of any of its symplectic leaves $\SL$ by the flow of any choice of modular vector field $u$:
\[
 Z = (\SL \times [0,k])/_{(x,0)\sim (\phi(x),k)},
\]
where $k$ is a certain positive real number and $\phi$ is the time-$k$ flow of $u$.
In the transverse direction to the symplectic leaves, all the modular vector fields flow with the same speed. This allows the following definition:
\begin{definition}\label{def:modperiod} Taking any modular vector field $u_{\mod}^\Omega$, the {\bf modular period} of $Z$ is the number $k$ such that $Z$ is the mapping torus
$$ Z = (\SL \times [0,k])/_{(x,0)\sim (\phi(x),k)},$$
and the time-$t$ flow of  $u_{\mod}^\Omega$ is translation by $t$ in the $[0, k]$ factor above.
\end{definition}

\subsubsection{ $b$-Poisson manifolds in the language of forms}
As it was done in \cite{Guillemin2012} we recall how \emph{forms with poles} can be introduced in a formal way for $b$-Poisson manifolds.
The idea is the following:
We  develop a concept which allows to extend the symplectic structure from $M\backslash Z$ to the whole manifold $M$. This singular form will be called a ``$b$-symplectic" form on $M$.

Let us first set the basic language that will be used: A \textbf{$b$-manifold} is a pair $(M^N,Z)$ of an oriented manifold $M$ and an oriented hypersurface $Z\subset M$. A \textbf{$b$-vector field} on a $b$-manifold $(M,Z)$ is a vector field which is tangent to $Z$ at every point $p\in Z$.
If $x$ is a local defining function for $Z$ on some open set $U\subset M$ and $(x,y_1,\ldots,y_{N-1})$ is a chart on $U$, then the set of $b$-vector fields on $U$ is a free $C^\infty(M)$-module with basis
$$(x \frbd{x}, \frbd{y_1},\ldots, \frbd{y_N}). $$
 By virtue of the Serre-Swan theorem\footnote{
R. \ Swan,   \emph{Vector Bundles and Projective Modules}, Transactions of the American Mathematical Society 105, (2), 264--277, (1962).} there exists a vector bundle associated to this module. This vector bundle is called the \textbf{$b$-tangent bundle} and denote it $^b TM$.

 The \textbf{$b$-cotangent bundle}  $^b T^*M$  of $M$ is defined to be the vector bundle dual to $^b TM$. Let us now describe the set of $b$-forms as sections of powers of this bundle.

For each $k>0$, let $^b\Omega^k(M)$ denote the space of \textbf{$b$-de Rham $k$-forms}, i.e., sections of the vector bundle $\Lambda^k(^b T^*M)$. The usual space of de Rham $k$-forms sits inside this space in a natural way; for $f$ a defining function of $Z$ every $b$-de Rham $k$-form can be written as
\begin{equation}\label{eq:bDeRham}
\omega=\alpha\wedge\frac{df}{f}+\beta, \text{ with } \alpha\in\Omega^{k-1}(M) \text{ and } \beta\in\Omega^k(M).
\end{equation}

The decomposition given by formula \eqref{eq:bDeRham} enables us to extend the exterior $d$ operator to $^b\Omega(M)$ by setting
$$d\omega=d\alpha\wedge\frac{df}{f}+d\beta.$$
The right hand side is well defined and agrees with the usual exterior $d$ operator on $M\setminus Z$ and also extends smoothly over $M$ as a section of $\Lambda^{k+1}(^b T^*M)$. Note that $d^2=0$, which allows us to define the complex of $b$-forms, the $b$-de Rham complex.

 The cohomology associated to this complex is called \textbf{$b$-cohomology} and it is denoted by \textbf{$^b H^*(M)$}.

 A special class of closed $2$-forms of this complex are $b$-symplectic forms as defined in \cite{Guillemin2012},
 \begin{definition} Let $(M^{2n},Z)$ be a $b$-manifold and $\omega\in\,^b\Omega^2(M)$ a closed $b$-form. We say that $\omega$ is \textbf{$b$-symplectic} if $\omega_p$ is of maximal rank as an element of $\Lambda^2(\,^b T_p^* M)$ for all $p\in M$.
\end{definition}

\begin{remark}Instead of working with $b$-Poisson structures we can dualize them and work with $b$-forms. In that sense, a $b$-symplectic form is just a symplectic form modeled over a different Lie algebroid (the $b$-cotangent bundle instead of the cotangent bundle).
\end{remark}

 In order for the $b$-complex to admit a Poincar\'{e} lemma, it is convenient to enlarge the set of smooth functions and consider the set of \textbf{$b$-functions} $^b C^\infty(M)$, which consists of functions with values in $\R \cup \{\infty\}$ of the form
$$c\,\textrm{log}|f| + g,$$
 where $c \in \mathbb{R}$, $f$ is a defining function for $Z$, and $g$ is a smooth function. For the sake of simplicity, from now on we identify $\R$ with the completion $\R \cup \{\infty\}$. The differential operator $d$ can be defined on this space in the obvious way:
$d(c\,\textrm{log}|f| + g):= \frac{c \, df}{f} + d g \in\, ^b\Omega^1(M),$
where $d g$ is the standard de Rham derivative.

Once the differential of a $b$-form has been defined, the Lie derivative of $b$-forms can be defined via the {\bf Cartan formula}:
 \begin{equation}\label{bcartan}
  \SL_X \omega = \iota_X (d \omega) + d (\iota_X \omega) \in ^b\Omega^k(M),
 \end{equation} where $\omega\in ^b\Omega^k(M)$ and $X$ is a $b$-vector field.

The following theorem shows how the $b$-cohomology can be related to De Rham cohomology:

\begin{theorem}[{\bf{Mazzeo-Melrose}}]\label{thm:mazzeomelrose}
The $b$-cohomology groups of $M^{2n}$ satisfy
$$^b H^*(M)\cong H^*(M)\oplus H^{*-1}(Z).$$
\end{theorem}

 In \cite{Guillemin2012}, it was proved that $b$-cohomology of a $b$-symplectic manifold is indeed isomorphic to its Poisson cohomology. The modular class  of the vector field transverse to the symplectic foliation in $Z$ can be identified with the component in $H^{1}(Z)$ defined by the degree 2 class of the $b$-symplectic form under the Mazzeo-Melrose isomorphism of Theorem \ref{thm:mazzeomelrose} and it is a Poisson invariant of the manifold. Many classification theorems (like the ones for toric actions in \cite{Guillemin2013}) can be interpreted in the light of this theorem.

\subsubsection{Hamiltonian $\T^n$-actions on $b$-symplectic manifolds}\label{sec:tractions}

Hamiltonian $\T^n$-actions will play a key role in the definition of the cotangent model for $b$-symplectic manifolds.  These actions were studied in \cite{Guillemin2013}. In this section we recall some definitions and results.

Hamiltonian vector fields are well-defined for smooth functions via the standard Poisson formalism. In \cite{Guillemin2013} we extended this concept to $b$-functions:

\begin{definition}[{\bf $b$-Hamiltonian vector field}] Let $(M,\omega)$  be a $b$-symplectic manifold. Given a $b$-function $H\in ^b \! C^\infty(M)$ we denote by $X_H$ the (smooth) vector field satisfying $\iota_{X_H} \omega = -dH.$
\end{definition}

Obviously, the flow of a $b$-Hamiltonian vector field preserves the $b$-symplectic form and hence the Poisson structure, so $b$-Hamiltonian vector fields are in particular Poisson vector fields.

\begin{definition}\label{def:bhamaction} An action of $\mathbb{T}^r$ on a $b$-symplectic manifold $(M^{2n},\omega)$ is \textbf{Hamiltonian} if for all $X,Y\in\mathfrak{t}$:
\begin{itemize}
\item the $b$-one-form $\iota_{X^\#}\omega$ is exact;
\item $\omega(X^\#, Y^\#)=0$.
\end{itemize}
\end{definition}
Here, $\mathfrak{t}$ denotes the Lie algebra of $\T^r$ and $X^\#$ is the fundamental vector field of $X$. The primitive of the exact $b$-one-form $\iota_{X^\#}\omega$ is defined via the moment map $\mu:M\to \mathfrak{t}^\ast$: $\iota_{X^\#} \omega|_p = d \langle \mu(p),X \rangle.$ In other words, $X^\#$ is the $b$-Hamiltonian vector field of $-\langle \mu(p),X \rangle$.

When a $b$-function $f\in C^\infty(M)$ is expressed as $c\log|y|+g$ locally near some point of a component $Z'$ of $Z$, the number $c_{Z'}(f) := c \in \mathbb{R}$ is uniquely determined by $f$, even though the functions $y$ and $g$ are not and it is a Poisson invariant.
\begin{definition}[Modular weight] Given a Hamiltonian $\mathbb{T}^r$-action on a $b$-symplectic manifold, the \textbf{modular weight} of a connected component $Z'$ of $Z$ is the map
$v_{Z'}:\mathfrak{t}\to \mathbb{R}$
given by $v_{Z'}(X)=c_{Z'}(H_X)$. This map is linear and therefore we can regard it as an element of the dual of the Lie algebra $v_{Z'}\in\mathfrak{t}^*$. We denote the kernel of $v_{Z'}$ by $\mathfrak{t}_{Z'} \subset \mathfrak{t}$.
\end{definition}

\subsection{Integrable systems  on symplectic manifolds} Let $(M^{2n},\omega)$ be a symplectic manifold. An \textbf{integrable system} is given by $n$ functions $f_1,\ldots,f_n$ in involution with respect to the Poisson bracket associated to the symplectic form $\omega$ and which are functionally independent on a dense set. Recall that the Poisson bracket associated to $\omega$ is defined via
$$\{f,g\}:=\omega(X_f,X_g),\qquad f,g\in C^\infty(M),$$
where for a function $f$ the vector field $X_f$ is the {\bf Hamiltonian vector field} of $f$ defined by $\iota_{X_f} \omega = -df$.

The expression \emph{integrable} refers to integrability of the system of differential equations associated to a function $H$ which can be chosen as one of the commuting functions: Integrability of the system in the sense described above (also called Liouville integrability) is related to actual integration of the system by quadratures \cite{liouville}.

The local structure of a symplectic manifold is described by the Darboux theorem. In the context of integrable systems the Darboux-Carath\'{e}odory theorem states that we can find a special Darboux chart in which half the coordinate functions are the integrals of the system (locally around a point where the integrals are independent).

The theorem of Liouville-Mineur-Arnold goes one step further and establishes a {\it semi-local} result in a neighbourhood of a compact level set (``Liouville torus'') of the integrable system:

\begin{theorem}{\bf(Liouville-Mineur-Arnold)}\label{thm:liouvillearnold} Let $(M^{2n},\omega)$ be a symplectic manifold. Let $F=(f_1,\ldots,f_n)$ be an $n$-tuple of functions on $M$ which are functionally independent (i.e. $df_1\wedge\dots\wedge df_n\neq 0$) on a dense set and which are pairwise in involution. Assume that $m$ is a regular point\footnote{i.e. the differentials $df_i$ are independent at $m$.} of $F$ and that the level set of $F$ through $m$, which we denote by $\F_m$, is compact and connected.

Then $\F_m$ is a torus and on a neighbourhood $U$ of $\F_m$ there exist ${\R} $-valued smooth
  functions $(p_1,\dots, p_n )$ and $ {\R}/{\Z}$-valued smooth functions
  $({\theta_1},\dots,{\theta_n})$ such that:
  \begin{enumerate}
    \item The functions $(\theta_1,\dots,\theta_n,p_1,\dots,p_n )$ define a diffeomorphism
            $U\simeq\T^n\times B^{n}$.
    \item The symplectic structure can be written in terms of these coordinates as
    \begin{equation*}
       \omega=\sum_{i=1}^n d \theta_i \wedge dp_i.
    \end{equation*}
    \item The leaves of the surjective submersion $F=(f_1,\dots,f_{s})$ are given by the projection onto the
      second component $\T^n \times B^{n}$, in particular, the functions $f_1,\dots,f_s$ depend only on
      $p_1,\dots,p_n$.
  \end{enumerate}
The coordinates $p_i$ are called action coordinates; the coordinates $\theta_i$ are called angle coordinates.
\end{theorem}

\begin{remark}In physics, usually one of the integrals $f_i$ of Theorem \ref{thm:liouvillearnold} is the energy $H$, e.g. $f_1=H$, and motion is given by the flow of the Hamiltonian vector field of $H$. Statement (3) in Theorem \ref{thm:liouvillearnold} implies that $H$ is constant along the level sets of the functions $f_i$. Moreover, since $df_i(X_H)=\{f_i,H\}=0$, the vector field $X_H$ is tangent to the level sets. More precisely, in the action-angle coordinate chart, the flow of $X_H$ is linear on the invariant tori.
\end{remark}

\begin{remark}
The proof of this theorem can be found for instance in \cite{arnold} or in \cite{duistermaat}. Arnold's proof uses the so-called period integrals to define the action coordinates, $$p_i=\int_{\gamma_i} \alpha$$
\noindent with $\alpha$ a Liouville one-form for $\omega$ (i.e., $d\alpha=\omega$) and $\gamma_i$ a cycle of a Liouville torus. This formula can already be found in \cite{mineur}.
The proof in \cite{duistermaat} actively uses the existence of a Hamiltonian $\mathbb T^n$-action which is tangent to the fibers defined by the collection of first integrals $F=(f_1,\dots, f_n)$.
The proof by Duistermaat follows several steps: The first one is setting the topology of the fibration, then constructing a Hamiltonian $\mathbb T^n$-action (for which a uniformization of periods of the Hamiltonian vector fields is needed) and then using the techniques native to symplectic geometry such as the invariance of the integrable system by this action. The upshot of Duistermaat's proof is that it can be adapted to more general settings such as the case of regular Poisson manifolds.

\end{remark}
Many important examples of dynamical systems in physics are integrable. A first class of examples  is given by any $2$-dimensional Hamiltonian system with $dH \neq 0$ on a dense set, e.g. the mathematical pendulum. Another example is the two-body problem, i.e. a system consisting of two bodies which interact through a potential $V$ that depends only on their distance. This system has  configuration space $\R^3 \times \R^3$ and Hamiltonian $H( {q_1},  {q_2},  {p_1},  {p_2})= \frac{{p_1}}{2 m_1} + \frac{{p_2}}{2 m_2} + V (|{q_1}-{q_2}|)$ is integrable with first integrals the energy $H$, the total momentum ${p_1}+{p_2}$ and the total angular momentum ${q_1}\times {p_1} + {q_2}\times {p_2}$. We also mention a  rigid body fixed at its centre of gravity, which is a system with configuration space $SO(3)$. Its integrals are the energy and the norm of the total angular momentum. Many more examples appear in classical mechanics.

\subsection{ Integrable systems on Poisson manifolds}
The results in this section are contained in \cite{Laurent-Gengoux2010}. The main idea is to extend the Liouville-Mineur-Arnold theorem to the context of Poisson manifolds. Such an extension can be performed in a neighbourhood of regular Liouville tori of Poisson manifolds.

\begin{definition}Let $(M,\Pi)$ be a Poisson manifold of dimension $n$, where $\Pi$ has maximal rank $2r$. An \textbf{integrable system} on $M$ is a set of $s=n-r$  functions $F=(f_1,\dots,f_s)$ which pairwise Poisson commute  and whose differentials $df_i$ are linearly independent on a dense set.
\end{definition}

We denote the non-empty subset of $M$ where the differentials $d f_1,\dots,d f_s$ are independent by $\SU_F$; the set where $\Pi$ attains its maximal rank $2r$ is denoted $M_{r}$. A compact invariant subset of $\SU_F\cap M_{r}$ is called regular Liouville torus.  In \cite{Laurent-Gengoux2010} the following  action-angle theorem  in a neighbourhood of a regular Liouville torus is proved.
\begin{theorem}[{\bf Laurent-Miranda-Vanhaecke, \cite{Laurent-Gengoux2010}}]\label{thm:aa_poisson}Let $(M,\Pi)$ be a Poisson manifold of dimension $n$ and rank~$2r$ and let $F=(f_1,\dots,f_s)$  be an integrable system on $M$. Let $m\in \SU_F \cap M_r$ and suppose that the integral manifold $\SF_m$ of $X_{f_1},\dots,X_{f_s}$, passing through $m$, is compact. Then $\SF_m$ is a torus (Liouville torus) and on a neighbourhood $U$ of $\SF_m$ we can define coordinates $(\theta_1,\dots,\theta_r,p_1,\dots,p_{s}): U \to  \T^r \times B^s$
    such that
      $$
        \Pi=\sum_{i=1}^r\pp{}{\theta_i}\we\pp{}{p_i}
      $$
where the coordinates $p_1,\dots,p_s$ depend only on $F$.
The functions $\theta_1,\dots,\theta_r$ are called {angle coordinates}, the functions $p_1,$
$\dots,p_r$ are called {action coordinates} and the remaining functions
$p_{r+1},\dots,p_{s}$ are called {transverse coordinates}.
\end{theorem}

 A key ingredient in the proof of this theorem contained in \cite{Laurent-Gengoux2010} is the adaptation of Duistermaat's method to this context is  the control of the Poisson cohomology in a neighbourhood of a Liouville torus in order to define $\mathbb T^n$-Hamiltonian actions with orbits the fibers of the integrable system.


\subsection{Integrable systems on $b$-symplectic manifolds}

In  \cite{KMS} we introduced a definition of integrable systems for $b$-symplectic manifolds, where we allow the integrals to be $b$-functions. Such a ``$b$-integrable system'' on a $2n$-dimensional manifold consists of $n$ integrals, just as in the symplectic case. \begin{definition}\label{def:pbintegrable}  A  \textbf{$b$-integrable system} on a $2n$-dimensional $b$-symplectic manifold $(M^{2n},\omega)$ is a set of  $n$ pairwise Poisson commuting $b$-functions $F=(f_1,\ldots,f_{n-1},f_n)$   with
$df_1 \wedge \dots \wedge df_n\neq 0$  as a section of $\wedge^n  ({^b} T^*(M))$ on a dense subset of $M$ and on a dense subset of $Z$. A point in $M$ is {\bf regular} if the vector fields $X_{f_1}, \dots, X_{f_n}$ are linearly independent (as \emph{smooth} vector fields) at it. \end{definition}

Notice that if a point on $Z$ is regular, then at least one of the $f_i$ must be non-smooth there. On the set of regular points, the distribution given by $X_{f_1},\ldots, X_{f_n}$ defines a foliation $\SF$. We denote the integral manifold through a regular point $m\in M$ by $\SF_m$. If $\SF_m$ is compact, then it is an $n$-dimensional torus (also referred to as ``(standard) {\bf Liouville torus}''). Because the $X_{f_i}$ are b-vector fields and are therefore tangent to $Z$, any Liouville torus that intersects $Z$ actually lies inside $Z$.
Two ($b$-)integrable systems $F$ and $F'$ are called {\bf equivalent} if there is a map $\mu:\R^n \supset F(M) \to \R^n$ taking one system to the other: $F' = \mu \circ F$. We will not distinguish between equivalent integrable systems.\begin{remark} Near a regular point of $Z$, a $b$-integrable system on a $b$-symplectic manifold is equivalent to one of the type $F=(f_1,\ldots,f_{n-1},f_n)$, where $f_1,\ldots, f_{n-1}$ are $C^\infty$ functions and $f_n$ is a $b$-function. In fact, we may always assume that  $f_n = \log|t|$, where $t$ is a global defining function for $Z$.
\end{remark}

In analogy to the Liouville-Mineur-Arnold theorem, we have
\begin{theorem}[{\textbf{Kiesenhofer-Miranda-Scott \cite{KMS}}}]\label{thm:baa} Let  $\left(M, \omega, F \right)$
  be a $b$-integrable system with $F= (f_1, \dots, f_{n-1}, f_n = \log|t|)$, and let $m \in Z$ be a regular point for which the integral manifold containing $m$ is compact, i.e. a Liouville torus $\SF_m$. Then there exists an open neighbourhood $U$ of the torus $\SF_m$ and coordinates
  $$(\theta_1,\dots,\theta_n,p_1,\dots,p_{n}): U \to  \T^n \times B^n$$ such that
\begin{equation}
        \omega|_U =\sum_{i=1}^{n-1}  d\theta_i  \wedge dp_i + \frac{c}{p_n} d\theta_n \wedge dp_n,
\end{equation} where the coordinates $p_1,\dots,p_n$ depend only on $F$ and the number $c$ is the modular period of the component of $Z$ containing $m$.
\end{theorem}

The main novelty of this result compared to the action-angle coordinate theorem in \cite{Laurent-Gengoux2010} is that it is an action-angle theorem in a neighbourhood of a Liouville tori lying inside the set of points where the Poisson structure is non-regular whilst the theorem in \cite{Laurent-Gengoux2010} only concerns Lioville tori which are regular.

The gist of the proof contained in \cite{KMS} is the use of  a $b$-Hamiltonian $\mathbb T^n$-action tangent to the fibers of the integrable system. This is a non-trivial generalization of Duistermaat's techniques because the actions implied are no longer Hamiltonian in the standard sense but Poisson actions. The obstruction for these actions to be  honest Hamiltonian is given by the modular geometry of the Poisson manifold. The constant $c$ reflects this dependence on the modular geometry of the Poisson manifold which is totally encoded by the connected component of the critical hypersurface containing the Liouville torus under consideration. An interesting application of the action-angle coordinates for $b$-symplectic manifolds is a KAM theorem contained in \cite{KMS} we refer the reader to that paper for details.

\section{Cotangent lifts and $b$-cotangent lifts}\label{sec:general_cotlift}

In this section we work towards the definition of the models for integrable systems in symplectic and $b$-symplectic manifolds. The standard definitions of cotangent lifts are reviewed and generalized  to the context of $b$-symplectic manifolds.
\subsection{General facts about cotangent lifts}

 Let $G$ be a Lie group and let $M$ be any smooth manifold. Given a group action $\rho:G\times M\longrightarrow M$, we define its cotangent lift as the action on $T^\ast M$ given by $\hat{\rho_g}:=\rho^\ast_{g^-1}$ where $g\in G$. We then have a commuting diagram
\begin{align}
\begin{diagram}
\node{T^\ast M}  \arrow{e,t}{\hat{\rho_g}}
\arrow{s,l}{\pi}
\node{T^\ast M} \arrow{s,r}{\pi}\\
\node{ M}  \arrow{e,t}{\rho_g}  \node{M}
\end{diagram}
\end{align}
where $\pi$ is the canonical projection from $T^\ast M$ to $M$.

 The cotangent bundle $T^*M$ is a symplectic manifold endowed
with the symplectic form $\omega=-d\lambda$, where $\lambda$ is the Liouville one-form. The latter can be defined intrinsically:
\begin{equation}\label{liouvilleform}
 \langle \lambda_p, v\rangle:= \langle p, (\pi_p)_\ast (v)\rangle
\end{equation}
with  $v\in T(T^*M), p\in T^*M$.

A straightforward argument \cite{guilleminandsternberg} shows that the cotangent lift $\hat{\rho}$  is Hamiltonian with moment map $\mu:T^*M \to \mathfrak{g}^*$ given by
\begin{equation*}\label{eqn:lift}
\langle\mu(p),X \rangle := \langle \lambda_p ,X^\#|_{p} \rangle =\langle p,X^\#|_{\pi(p)}\rangle,
\end{equation*}
where  $p\in T^*M$, $X$ is an element of the Lie algebra $\mathfrak{g}$ and we use the same symbol $X^\#$ to denote the fundamental vector field of $X$ generated by the action on $T^\ast M$ or $M$.

An easy computation shows that the Liouville one-form is invariant under the action, i.e. $\hat \rho_g^\ast \lambda = \lambda$. It is known that invariance of $\lambda$ implies equivariance of the moment map $\mu$, meaning that
$$\mu \circ \hat \rho_g = Ad_{g^{-1}}^\ast \circ \mu .$$
A consequence is that the moment map is Poisson\footnote{cf. Proposition 7.1 in \emph{ A. Cannas da Silva and A. Weinstein, Geometric Models for Noncommutative Algebras Berkeley Mathematics Lecture Notes series, American Mathematical Society, 1999}.}.

We will also refer to this construction as the {\bf symplectic cotangent lift}.

\subsection{Symplectic cotangent lift of translations on the torus}\label{subsec:cot} In the special case where the manifold $M$ is a torus $\T^n$ and the group is $\T^n$ acting by translations, we obtain the following explicit structure: Let $\theta_1,\ldots,\theta_n$  be the standard ($S^1$-valued) coordinates on $\T^n$ and let
\begin{equation}\label{co}
\underbrace{\theta_1,\ldots,\theta_n}_{=:\theta}, \underbrace{a_1, \ldots, a_n}_{:=a}
\end{equation}
be the corresponding chart on $T^\ast \T^n$, i.e. we associate to the coordinates \eqref{co} the cotangent vector $\sum_i a_i d \theta_i \in T^\ast_\theta \T^n$.
The Liouville one-form, which we defined intrinsically above, is given in these coordinates by
$$ \lambda = \sum_{i=1}^n a_i d \theta_i $$
and its negative differential is the standard symplectic form on $T^\ast \T^n$:
\begin{equation} \omega_{can} = \sum_{i=1}^n d \theta_i \wedge d a_i .\label{eq:omegacan}\end{equation}
We denote by $\tau_\beta$ the translation by $\beta \in \T^n$ on $\T^n$; its lift to $T^\ast \T^n$ is given by
$$ \hat \tau_\beta: (\theta, a) \mapsto (\theta + \beta, a).$$
The moment map $\mu_{can}: T^\ast \T^n \to \mathfrak{t^\ast} $ of the lifted action with respect to the canonical symplectic form is
\begin{equation}\label{eq:mucan} \mu_{can}(\theta,a) = \sum_i a_i d\theta_i, \end{equation}
where the $\theta_i$ on the right hand side are understood as elements of $ \mathfrak{t^\ast}$ in the obvious way. Even simpler, if we identify $ \mathfrak{t^\ast}$ with $\R^n$ by choosing the standard basis $\frbd{\theta_1},\ldots, \frbd{\theta_n}$ of  $\mathfrak{t}$ then the moment map is just the projection onto the second component of $T^\ast \T^n \cong \T^n \times \R^n$.  We will adopt this viewpoint from now on. Note that the components of $\mu$ naturally define an integrable system on $T^\ast \T^n$.

\subsection{$b$-Cotangent lifts of $\T^n$}\label{subsec:bcot} As before, let $T^\ast \T^n$ be endowed with the standard coordinates $(\theta, a)$, $\theta \in \T^n$, $a \in \R^n$ and consider again the action on $T^\ast \T^n$ induced by lifting translations of the torus $\T^n$. We now want to view this action as a $b$-Hamiltonian action with respect to a suitable $b$-symplectic form. In analogy to the classical Liouville one-form we define the following  non-smooth one-form away from the hypersurface $Z=\{a_1 = 0\}$~:
$$c \log|a_1| d \theta_1 + \sum_{i=2}^n a_i d\theta_i.$$
The negative differential  of this form extends to a $b$-symplectic form on $T^\ast \T^n$, which we call the {\bf twisted $b$-symplectic form }on $T^\ast \T^n$ (we will explain the terminology below):
\begin{equation}\label{twistedform}
 \omega_{tw, c}:=\frac{c}{a_1} d\theta_1\wedge d a_1 + \sum_{i=2}^n d\theta_i\wedge da_i  .
\end{equation}
The moment map of the lifted action with respect to this $b$-symplectic form is then given by
\begin{equation}\label{eq:bmucan}\mu_{tw, c}:=(c \log|a_1|,a_2, \ldots, a_n),\end{equation}
where we identify $\mathfrak{t^\ast}$ with $\R^n$ as before.

We call this lift together with the $b$-symplectic form \eqref{twistedform} the {\bf twisted $b$-cotangent lift} with modular period $c$. Note that, in analogy to the symplectic case, the components of the moment map define a $b$-integrable system on $(T^\ast \T^n, \omega_{tw, c})$.

\begin{remark}\label{canonicalb} We use the term ``twisted $b$-symplectic form'' to distinguish our construction from the canonical $b$-symplectic form on $^b T^*M$, where $M$ is any smooth manifold. The latter is obtained naturally if we use the intrinsic definition of the Liouville one-form  \eqref{liouvilleform} in the $b$-setting (see e.g. \cite{nest}). More precisely, for $M$ a $b$-manifold, we define a $b$-form $\lambda$ on $^b T^*M$ via
 \begin{equation}\label{canonical_liouville}
  \langle \lambda_p, v\rangle:= \langle p, (\pi_p)_\ast (v)\rangle,
  \end{equation}
where $v\in ^b T(^b T^*M)$ and $p\in ^b T^*M$. The negative differential
$$\omega = - d \lambda$$
is the canonical $b$-symplectic form on $^b T^*M$.
Here, we view $^b T^*M$ as a $b$-manifold with hypersurface $\pi^{-1} (Z)$ where
$$\pi: \, ^b T^\ast M \to M$$ is the canonical projection. Choosing a local set of coordinates $x_1,\ldots, x_n$ on $M$, where $x_1$ is a defining function for $Z$ we have a corresponding chart
$$(x_1,\ldots,x_n,p_1,\ldots,p_n)$$
on $T^\ast M$, given by identifying the $2n$-tuple above with the $b$-cotangent vector
$$p_1 \frac{dx_1}{x_1} + \sum_{i = 2}^n p_i dx_i \in  ^b T_x^*M.$$
In these coordinates
$$\lambda = p_1 \frac{dx_1}{x_1} + \sum_{i = 2}^n p_i dx_i \in ^b\! T^\ast(^b T^*M).$$

Note that the singularity here is given by the coordinate $x_1$ on the base manifold whereas in our ``twisted'' construction it is given by a fiber coordinate, which is what we require for the description of $b$-integrable systems.
\end{remark}

\subsection{$b$-Cotangent lifts in the general setting}\label{sec:cotlift_general} Above we focused on the case where the manifold $M$ is a torus and the action is by rotations of the torus on itself, since this is the model that describes ($b$-)integrable systems semilocally around a Liouville torus.

To obtain a wider class of examples, we now consider  any manifold $M$ and the action of any Lie group $G$ on $M$:
\begin{equation}\label{cotlift}
\rho:G\times M \to M : (g,m)\mapsto \rho_g(m).
\end{equation}
As described in Section \ref{sec:general_cotlift} we can lift the action to an action $\hat{\rho}$ on $T^\ast M$, which is Hamiltonian with respect to the standard symplectic structure on  $T^\ast M$. We want to investigate modifications of this construction, which lead to Hamiltonian actions on $b$-symplectic manifolds.

\subsubsection{\bf{Canonical $b$-cotangent lift.}} Connecting with Remark \ref{canonicalb}, assume that $M$ is an $n$-dimensional $b$-manifold with critical hypersurface $Z$. Instead of $T^\ast M$ consider the $b$-cotangent bundle $^b T^\ast M$ endowed with the canonical $b$-symplectic structure as described in the remark. Moreover, assume that the action of $G$ on $M$ preserves the hypersurface $Z$, i.e. $\rho_g$ is a $b$-map for all $g\in G$. Then the lift of $\rho$ to an action on  $^b T^\ast M$ is well-defined:
$$\hat{\rho}:G\times ^b \! T^\ast M \to ^b \! T^\ast M : (g,p)\mapsto \rho^\ast_{g^{-1}}(p).$$
We call this action on $^b T^\ast M$, endowed with the canonical $b$-symplectic structure, the {\bf canonical $b$-cotangent lift}.

\begin{proposition}\label{prop:cancotlift}
The canonical $b$-cotangent lift is Hamiltonian with equivariant moment map given by
 \begin{equation}\label{momentmap_canonical}\mu:\, ^b \! T^*M \to { \mathfrak{g}^*}, \quad
\langle\mu(p),X \rangle := \langle \lambda_p ,X^\#|_p \rangle =\langle p,X^\#|_{\pi(p)}\rangle,
\end{equation}
where $p\in \! ^b T^\ast M$, $X\in \mathfrak{g}$, $X^\#$ is the fundamental vector field of $X$ under the action on $^b T^\ast M$ and the function $\langle \lambda,X^\# \rangle$ is smooth because  $X^\#$  is a $b$-vector field.
\end{proposition}

\begin{proof}
The proof of Equation \eqref{momentmap_canonical} for the moment map is exactly the same as in the symplectic case: Using the implicit definition of $\lambda$, Equation \eqref{canonical_liouville}, we show that $\lambda$ is invariant under the action:
\begin{align*}\langle (\hat \rho_g^\ast \lambda)_p,v \rangle &= \langle \lambda_{\hat \rho_g(p)}, (\hat \rho_g)_\ast v \rangle =
 \langle \hat \rho_g(p) , (\pi_{\hat \rho_g(p)})_\ast ((\hat \rho_g)_\ast v )\rangle  = \\
 \\ &= \langle \rho^\ast_{g^{-1}}(p), (\rho_{g^{-1}})_\ast((\pi_p)_\ast(v)) \rangle= \langle p,(\pi_p)_\ast(v)\rangle .
 \end{align*}
 In going from the first to the second line we have used the definition of $\hat \rho$ and applied the chain rule to $\pi_{\hat \rho_g(p)} \circ \hat \rho_g = \rho_{g^{-1}} \circ \pi_p.$

Hence we have
$\SL_{X^\#}\lambda = 0$
and applying the Cartan formula for $b$-symplectic forms, Equation \eqref{bcartan}, we obtain
$$  \iota_{X^\#} \omega_p = - \iota_{X^\#} d\lambda_p = d(\iota_{X^\#} \lambda_p),$$
which proves the expression for the moment map stated above.

Equivariance of $\mu$ is a consequence of the invariance of $\lambda$:
\begin{align*}
\langle &(Ad_{g^{-1}}^\ast  \circ \mu)(p), X  \rangle = \langle \mu(p), Ad_{g^{-1}} X \rangle = \langle \lambda_p, \underbrace{(Ad_{g^{-1}} X )^\#}_{=  (\hat \rho_g)_\ast X^\#}|_p\rangle = \\
& = \langle \hat \rho_g^\ast \lambda_p, X^\#|_{\hat \rho_{g^{-1}}(p)} \rangle = \langle \lambda_{ \hat \rho_{g^{-1}}(p)}, X^\# |_{\hat \rho_{g^{-1}}(p)} \rangle =  \langle \mu( \hat \rho_{g^{-1}}(p)), X \rangle
\end{align*}
for all $g \in G$, $X\in \mathfrak{g}$, $p\in  T^\ast M$, where in the first equality of the second line we have used that $\lambda$ is invariant.
\end{proof}

\begin{remark} The condition that the action preserves $Z$ means that all fundamental vector fields are tangent to $Z$ and therefore at a point in $Z$ the maximum number of independent fundamental vector fields is $n-1$. This means that the moment map of such an action never defines a $b$-integrable system on $^b T^\ast M$ since this would require $n$ independent functions.
\end{remark}

\subsubsection{\bf{Twisted $b$-cotangent lift.}}

We have already defined the twisted $b$-cotangent lift on the cotangent space of a torus $T^\ast \T^n$ in Section \ref{subsec:bcot}. In particular, on $T^\ast S^1$ with standard coordinates $(\theta, a)$ we have the logarithmic Liouville one-form
$\lambda_{tw,c} = \log |a| d\theta$ for $a\neq 0$.

Now consider any $(n-1)$-dimensional manifold $N$ and let $\lambda_N$ be the standard Liouville one-form on  $T^\ast N$. We endow the product $T^\ast (S^1 \times N) \cong T^\ast S^1 \times T^\ast N$ with the product structure $\lambda:= (\lambda_{tw,c}, \lambda_N)$ (defined for $a\neq 0$). Its negative differential $\omega = -d\lambda$ is a $b$-symplectic structure  with critical hypersurface given by $a=0$.

Let $K$ be a Lie group acting on $N$ and consider the component-wise action of $G:=S^1\times K$ on $M:=S^1 \times N$ where $S^1$ acts on itself by rotations. We lift this action to $T^\ast M$ as described in the beginning of this section. This construction, where $T^\ast M$ is endowed with the $b$-symplectic form $\omega$, is called the \textbf{twisted $b$-contangent lift}.

If $(x_1,\ldots,x_{n-1})$ is a chart on $N$ and $(x_1,\ldots,x_{n-1},y_1,\ldots,y_{n-1})$ the corresponding chart on $T^\ast N$ we have the following local expression for $\lambda$
\begin{equation*}\label{logliouvilleform}
\lambda = \log |a| d\theta + \sum_{i=1}^{n-1} y_i dx_i .
\end{equation*}

Just as in the symplectic case and in the case of the  canonical $b$-cotangent lift, this action is Hamiltonian with moment map given by contracting the fundamental vector fields with $\lambda$:

\begin{proposition}\label{prop:twmoment} The twisted $b$-cotangent lift on $M=S^1\times N$ is Hamiltonian with equivariant moment map $\mu$ given by
 \begin{equation}
\langle\mu(p),X \rangle := \langle \lambda_p ,X^\#|_p \rangle,
\end{equation}
where $X^\#$ is the fundamental vector field of $X$ under the action on $T^\ast M$.
\end{proposition}
\begin{proof}
As in the proof of Proposition \ref{prop:cancotlift}, we show that the action preserves the logarithmic Liouville one-form $\lambda=(\lambda_{tw,c}, \lambda_N)$. Since the action splits this amounts to showing invariance of $\lambda_{tw,c}$ under $S^1$; the invariance of $\lambda_N$ under $K$ is the classical symplectic result. The former is easy to see:
$$(\hat{\tau})^\ast_\varphi \lambda_{tw,c} =  \log |a\circ \hat{\tau}_\varphi| d(\underbrace{\theta \circ \hat{\tau}_\varphi}_{= \theta + \varphi})= \log |a| d\theta,$$
where $\tau$ is the action of $S^1$ on itself by rotations and $\varphi \in S^1$. This shows that $\SL_{X^\#}\lambda = 0$  and as before we conclude the proof by using Cartan's formula.
\end{proof}

\begin{remark} A special case of a manifold $S^1\times N$ is a cylinder $\T^k \times \R^{n-k}$. We will use this construction in Section \ref{sec:affine}.
\end{remark}

In order to compute the moment map it is convenient to observe that the expression $\langle \lambda, X^\# \rangle$ remains unchanged when we replace the fundamental vector field $ X^\#$ of the action on $T^\ast M$ by any vector field on $T^\ast M$ that projects to the same vector field on $M$ (namely the fundamental vector field of the action on $M$). This follows immediately from the definition of $\lambda$.

\section{Cotangent models for integrable systems}\label{sec:cotmodels}
In this section we give cotangent models for integrable systems on symplectic and $b$-symplectic manifolds as it was done for Hamiltonian actions in \cite{marle} and \cite{guilleminandsternberg2}.
We write a ($b$-)integrable system as a triple $(M, \omega, F)$ where $M$ is a manifold, $\omega$ a ($b$-)symplectic form and $F$ the set of integrals.
With the notation introduced in Subsections \ref{subsec:cot} and \ref{subsec:bcot} we define the following models of integrable systems, which we will use below to give a semilocal description of integrable and $b$-integrable systems.
\begin{enumerate}\item \begin{equation}(T^\ast \T^n)_{can} := (T^\ast \T^n, \omega_{can}, \mu_{can})\label{eq:can}\end{equation} with $\omega_{can}, \mu_{can}$ defined by Equations (\ref{eq:omegacan}) and (\ref{eq:mucan}) respectively.
 \item  \begin{equation}(T^\ast \T^n)_{tw, c} := (T^\ast \T^n, \omega_{tw, c}, \mu_{tw, c})\end{equation} with $\omega_{tw, c}, \mu_{tw, c})$ defined by Equations (\ref{twistedform}) and (\ref{eq:bmucan}) respectively.
 \end{enumerate}
We say that two ($b$-)integrable systems $(M_1, \omega_1, F_1)$ and $(M_2, \omega_2, F_2)$ are {\bf equivalent} if there exists a Poisson diffeomorphism $\psi$ and a map $\varphi:  \R^s \to  \R^s$ such that the following diagram commutes:
\begin{align*}
  \begin{diagram}
    \node{(M_1, \omega_1)}\arrow{e,t}{\psi}\arrow{se,b}{F_1}\node{(M_2, \omega_2)}\arrow{s,r}{\varphi \circ  F_2}\\
    \node[2]{\R^{s}}
  \end{diagram}
\end{align*}


\subsection{Symplectic case} We restate the Liouville-Mineur-Arnold theorem (Theorem \ref{thm:liouvillearnold}) in terms of the symplectic cotangent model:

\begin{theorem} Let $F=(f_1,\ldots,f_n)$ be an integrable system on the symplectic manifold $(M,\omega)$. Then semilocally around a regular Liouville torus the system is equivalent to the cotangent model $(T^\ast \T^n)_{can}$ restricted to a neighbourhood of the zero section $(T^\ast \T^n)_0$ of $T^\ast \T^n$.
\end{theorem}
\begin{proof}
Let $\ST$ be a regular Liouville torus of the system.
The action-angle coordinate theorem (Theorem \ref{thm:liouvillearnold}) implies that there exists a neighbourhood $U$ of $\ST$ and a symplectomorphism
 $$ \psi: U \to  (\T^n \times B^n, \omega_{can})$$
such that the ``action coordinates'', i.e. the projections onto $ B^n$, depend only on the integrals $f_1,\ldots,f_n$, hence their composition with $\psi$ yields an equivalent integrable system on $U$. We know that the projections onto $B^n$ correspond to the moment map $\mu_{can}$ of the cotangent lifted action on $T^\ast \T^n \cong \T^n \times \R^n$ (restricted to $\T^n \times B^n$ and understood with respect to the canonical basis on $\mathfrak{t^\ast}$), hence we can write
\begin{align*}
  \begin{diagram}
    \node{U}\arrow{e,t}{\psi}\arrow{se,b}{F}\node{(T^\ast \T^n, \omega_{can})}\arrow{s,r}{\varphi \circ  \mu}\\
    \node[2]{\R^{n}}
  \end{diagram}
\end{align*}
where $\varphi$ is the map that establishes the dependence of the action coordinates on $f_1,\ldots,f_n$.
\end{proof}

\subsection{$b$-symplectic case} The model of twisted $b$-cotangent lift lets us express the action-angle coordinate theorem for $b$-integrable systems in the following way:
\begin{theorem} Let $F=(f_1,\ldots,f_n)$ be a $b$-integrable system on the $b$-symplectic manifold $(M,\omega)$.  Then semilocally around a regular Liouville torus $\ST$, which lies inside the exceptional hypersurface $Z$ of $M$, the system is equivalent to the cotangent model $(T^\ast \T^n)_{tw, c} $ restricted to a neighbourhood of $(T^\ast \T^n)_0$. Here $c$ is the modular period of the connected component of $Z$ containing  $\ST$.
\end{theorem}
\begin{proof}
The proof is the same as above using the action-angle coordinate theorem for $b$-integrable systems (Theorem \ref{thm:baa}): Around the Liouville torus $\ST$ we have a Poisson diffeomorphism
 $$ \psi: U \to  \T^n \times B^n$$
taking the $b$-symplectic form on $U$ to
$$\sum_{i=1}^{n-1}  d\theta_i  \wedge dp_i + \frac{c}{p_n} d\theta_n \wedge dp_n,$$
where $(\theta_1,\ldots, \theta_n, p_1,\ldots,p_n)$  are the standard coordinates on $\T^n \times B^n$, and such that $p_1,\ldots,p_n$ only depend on the integrals. Hence in the language of Section \ref{subsec:bcot} we have a commuting diagram
\begin{align*}
  \begin{diagram}
    \node{U}\arrow{e,t}{\psi}\arrow{se,b}{F}\node{(T^\ast \T^n, \omega_{tw, c})}\arrow{s,r}{\varphi \circ  \mu_{tw, c}}\\
    \node[2]{{\R^{n}}}
  \end{diagram}
\end{align*}
\end{proof}
\begin{remark} Observe that the model depends strongly on the modular class of the Poisson manifolds. Recall that via the Mazzeo-Melrose theorem and the fact that $b$-cohomology is isomorphic to Poisson cohomology, the constant  $c$ can be seen as a Poisson invariant.
\end{remark}

\section{ Constructing examples on $b$-symplectic manifolds}\label{sec:examplesb}

As an application of the models above we can construct examples of ($b-$)integrable systems:

\begin{theorem}\label{thm:cotlift_intsys} Let $M$ be a smooth manifold of dimension $n$ and let $G$ be a $n$-dimensional abelian Lie group acting on $M$ effectively. Pick a basis $X_1,\ldots, X_n$ of the Lie algebra of $G$. Consider the moment map
$\mu: T^\ast M \to {\mathfrak{g}^\ast}$
of one of the following Hamiltonian actions:
\begin{enumerate}
\item the (symplectic) cotangent lift on $T^\ast M$.
\item the twisted $b$-cotangent lift on $T^\ast M$, where we assume that $M=S^1 \times N$ and $G=S^1\times K$ for $N$ an $n-1$ dimensional manifold and $K$ a Lie algebra and that the action splits with $S^1$ acting on itself by rotations.
\end{enumerate}
Then the components of the moment map with respect to the basis $X_i$ define an (1) integrable resp. (2) $b$-integrable system on $T^\ast M$.
\end{theorem}

\begin{proof}Denote the components of the moment map by $f_i:= \langle \mu, X_i \rangle$. Effectiveness of the action implies that the $f_i$ are  linearly independent everywhere. Moreover, since $\mu$ is a Poisson map and the elements $X_i$ commute, we obtain $\{f_i, f_j\}  = 0. $
\end{proof}

\subsection{\textbf{The geodesic flow}} A special case of a $\T^n$-action is obtained in the case of a Riemannian manifold $M$ which is assumed to have the property that all its geodesics are closed, so-called {\bf P-manifolds}. Then the geodesics admit a common period (see e.g. \cite{besse}, Lemma 7.11); hence their flow induces an $S^1$-action on $M$ and we can use the twisted $b$-cotangent lift to obtain a $b$-Hamiltonian $S^1$-action on $T^\ast M$. The moment map then corresponds to a {\it non-commutative} $b$-integrable system on $T^\ast M$, which is a generalization of the systems studied here and will be explored in a future work. In dimension two, examples of P-manifolds are Zoll and Tannery surfaces (see Chapter 4 in \cite{besse}). Given an $S^1$-action on such a surface, via the cotangent lift we immediately obtain examples of ($b$-)integrable systems on its cotangent bundle.

\subsection{\textbf{Affine manifolds}}\label{sec:affine} A smooth manifold $M$ is called \textbf{flat} if it admits a flat (i.e. zero curvature) connection. It is called \textbf{affine} if moreover the connection is torsion-free.

 It is well-known that a simply connected flat manifold is parallelizable, i.e. it admits a basis of  vector fields that are everywhere independent. Such a basis is called parallel. The relation between flatness (in the sense that the curvature is zero) and parallelizability was studied in  \cite{thorpe}. We are not assuming that the affine manifold is compact.

Bieberbach \cite{bieberbach} proved in 1911 that any complete affine Riemannian manifold is a finite quotient of $\mathbb R^k\times \mathbb T^{n-k}$.

 \begin{theorem} Let $M$ be a cylinder $\mathbb R^k\times \mathbb T^{n-k}$. Then for any choice of parallel basis $X_1,\dots X_n$, we obtain a ($b$-)integrable system on $T^*M$.
 \end{theorem}
\begin{proof}

Let $X_1,\dots, X_n$ be a global basis of parallel vector fields. Since the torsion of the connection is zero and the vector fields $X_i$ are parallel, the expression $\nabla_{X_i} X_j- \nabla_{X_j} X_i-[X_i, X_j]=T^{\nabla}(X_i, X_j)=0$ yields $[X_i, X_j]=0$. In other words, the flows of the vector fields commute. Let us denote by $\Phi_{X_{j}}^{s_j}$  the $s_j$-time flow of the vector field $X_j$.
 Since the manifold is complete, the joint flow of the vector fields $X_i$ then defines an $\mathbb R^n$-action\footnote{Depending on the topology of the fiber, this action may descend to a $\mathbb T^n$-action or more generally to a $\mathbb R^k\times\mathbb T^{n-k}$-action.},
\begin{align*}
\Phi :\R^n \times  M  &\to M \\
 \big((s_1,\ldots,s_n),(x)\big) &\mapsto \Phi_{X_{_1}}^{s_1}\circ\dots\circ \Phi_{X_{n}}^{s_n}((x)).
\end{align*} By the construction defined in Section \ref{sec:cotlift_general} we obtain a ($b$-)Hamiltonian action on $T^*M$ and the components of the moment map of this action define a ($b$-)integrable system (Theorem \ref{thm:cotlift_intsys}).
\end{proof}

\begin{remark}We proved the above result only for cylinders $\mathbb R^k\times \mathbb T^{n-k}$. It would be interesting to explore whether a similar construction is possible for finite quotients of $\mathbb R^k\times \mathbb T^{n-k}$, which by Bieberbach's result would correspond to all complete affine Riemannian manifolds.
\end{remark}

\begin{remark}Even if this procedure yields examples of $b$-integrable systems on non-compact manifolds, we may consider Marsden-Weinstein reduction to obtain compact examples. Reduction in the $b$-setting is already plotted in \cite{Guillemin2013} for abelian groups. The general scheme follows similar guidelines.
\end{remark}

\section{$b$-integrable systems with singularities}\label{sec:examples}

The purpose of this section is to use the ideas described in this paper to depict models for singular integrable systems on ($b$-)symplectic manifolds. When we say that the system is singular on a set we mean that the Hamiltonian vector fields $X_{f_i}$ are not independent there. For singular integrable system, there are some subtleties concerning the distribution defined by it (cf. \cite{miranda, miranda1}). When we say the foliation is defined by a set of first integrals, we mean that the distribution defined by the Hamiltonian vector fields of a set of functions is the same.

In this framework we will introduce non-degenerate integrable systems as twisted $b$-cotangent lifts of some action. This section is intended as an invitation to the study of singularities of integrable systems.

\subsection{The harmonic oscillator} Let us consider the $2$-dimensional harmonic oscillator, i.e. the coupling of two simple harmonic oscillators. The configuration space is $\mathbb R^2$ with standard coordinates $x=(x_1,x_2)$. The phase space is $T^*(\mathbb R^2)$ endowed with symplectic form $\omega=dx_1\wedge dy_1+dx_2\wedge dy_2$.
$H$ is the sum of potential and kinetic energy, $$H=\frac{1}{2}(y_1^2+y_2^2)+ \frac{1}{2}(x_1^2+x_2^2)$$ The level set $H=h$ is a sphere $S^3$. Considering the rotational symmetry on this sphere we fined another first integral.  The angular momentum
$L=x_1y_2-x_2y_1$  corresponds to the action by rotations lifted to the cotangent bundle. Its Hamiltonian vector field  is $X_L=(-x_2,x_1,-y_2,y_1)$. This yields $X_L(H)=\{L, H\}=0$
thus proving integrability of  the system.

To construct a $b$-integrable system, we consider $S^1 \times S^1$ acting on $M:=S^1 \times \R^2$, where the first $S^1$ component acts on itself by rotations and the second one acts on $\R^2$ by rotations. We lift the action to $T^\ast M$, which we endow with the twisted $b$-symplectic form, see Section \ref{sec:cotlift_general}. The moment map with respect to the standard basis of the Lie algebra of $S^1 \times S^1$ is then given by
$$ T^\ast M \cong S^1 \times \R^2 \times \R \times \R^2 \to \R^2:
(\theta, x_1, x_2, a, y_1, y_2) \mapsto (\log|a|, x_1 y_2 - x_2 y_1).$$
Note that the second component is the angular momentum $L$. We can complete these two functions to a $b$-integrable system by adding the energy $H$, i.e. the system is given by $(\log|a|, L,H)$.

\subsection{Hyperbolic singularities}\label{sec:hyp} Consider the group $G:= S^1 \times \R^+$ acting on $M:=S^1 \times \R$  in the following way:
$$(\varphi, g)\cdot (\theta, x):=(\theta + \varphi, g x),$$
i.e. on the $S^1$ component we have rotations and on the $\R$ component we have multiplications.
Then the Lie algebra basis $(\pp{}{\theta},  \pp{}{g})$
induces the following fundamental vector fields on $M$:
$$X_1:= \pp{}{\theta}, \qquad X_2:=x \pp{}{x}.$$
As defined in Section \ref{sec:cotlift_general} we consider the twisted $b$-cotangent lift on $T^\ast M$, i.e. the $b$-symplectic structure $\omega = - d\lambda$ where
 $$\lambda := \log |p| d\theta + y dx$$
and $(\theta, p, x, y)$ are the standard coordinates on $T^\ast M$.
As we showed in Proposition \ref{prop:twmoment}, the lifted action on $T^\ast M$ is $b$-Hamiltonian with moment map given by
$\mu:=(f_1,f_2)$:
\begin{align*}
f_1 &= \langle \lambda, X_1^\# \rangle =  \log|p|, \\
f_2 &= \langle \lambda, X_2^\# \rangle =  x y.
\end{align*}
This type of singular $b$-integrable system is known as {\it hyperbolic singularity}.

\begin{definition} Let $(f_1,f_2)$ be a $b$-integrable system and let $p\in M$ be a point where the system is singular, we say that the singularity is of \textbf{hyperbolic type} if there is a chart $(t,z, x, y)$ centred at $p$ such that the critical hypersurface of $\omega$ is locally around $p$ given by $t=0$ and the integrable system is determined by
 $$f_1 = c \log|t|, \quad  f_2 = x y.$$
\end{definition}

\subsection{Focus-focus singularities}\label{sec:focus} Consider the group $G:= S^1 \times \R^+ \times S^1$ acting on $M:=S^1 \times \R^2$  in the following way:
$$(\varphi, a, \alpha)\cdot (\theta, x_1,x_2):=(\theta + \varphi, a R_\alpha (x_1, x_2)),$$
where $R_\alpha$ is the matrix corresponding to rotation by $\alpha$ in the plane. In other words, on $\R^2$ we have $\R^+$ acting by radial contractions/expansions and $S^1$ acting by rotations.

Using the coordinates above, the Lie algebra basis $(\pp{}{\theta},  \pp{}{a}, \pp{}{\alpha})$
induces the following fundamental vector fields on $M$:
$$X_1:= \pp{}{\theta}, \qquad X_2:=x_1 \pp{}{x_1} + x_2\pp{}{x_2}, \qquad X_3:= x_1 \pp{}{x_2} - x_2 \pp{}{x_1}.$$

As above, we consider the twisted $b$-cotangent lift. The logarithmic Liouville one-form is
$$\lambda := \log |p| d\theta + y_1 dx_1 + y_2 d x_2$$
and the moment map is, according to Proposition \ref{prop:twmoment}, given by
$\mu:=(f_1,f_2,f_3)$ with
\begin{align*}
f_1 &= \langle \lambda, X_1^\# \rangle =  \log|p|, \\
f_2 &= \langle \lambda, X_2^\# \rangle =  x_1 y_1 + x_2 y_2, \\
f_3 &= \langle \lambda, X_3^\# \rangle =  x_1 y_2 -y_1 x_2.
\end{align*}

In the theory of singular integrable systems on symplectic manifolds, the last two components define the well-known focus-focus singularity if we extend the manifold $M$ to include points with $(x_1,x_2)=0$.
\begin{definition}  Let $(f_1,f_2,f_3)$ be a $b$-integrable system and let  $p\in M$ be a point where the system is singular. We say that the singularity is of \textbf{focus-focus type} if there is a chart $(t,z, x_1, y_1,  x_2, y_2)$ centered at $p$ such that the critical hypersurface of $\omega$ is locally around $p$ given by $t=0$ and the integrable system is given by
 $$f_1 = c \log|t|, \quad  f_2 = x_1 y_1 + x_2 y_2,\quad f_3 = x_1 y_2 - y_1 x_2.$$
\end{definition}

\begin{remark} By simplifying the examples above and eliminating the ``$b$-part" we may also introduce non-degenerate singularities of integrable systems on symplectic manifolds in the sense of  \cite{eliasson1}, \cite{mirandazung}, \cite{miranda}, \cite{miranda1} and view them as cotangent lifts.
We may define non-degenerate singularities of integrable systems on $b$-symplectic manifolds as Cartan subalgebras of $\mathfrak{sp}(2n-1,\mathbb R)\oplus \mathbb R$.
\end{remark}

\begin{remark}  We may obtain general $(0,k_h,k_f)$-Williamson type\footnote{The Williamson type of a non-degenerate singularity of an integrable system on a symplectic manifold is given by a triple $(k_e,k_h,k_f)$ and is an invariant of the orbit containing the singularity \cite{mirandazung}. This concept can be generalized to $b$-symplectic manifolds.} singularities of integrable systems and view them as $b$-cotangent lifts by coupling the examples in Subsections \ref{sec:hyp} and \ref{sec:focus}.
\end{remark}

\section{Cotangent models for Poisson manifolds}

Given two integrable systems $(M_1, \Pi_1, F_1)$ and $(M_2, \Pi_2, F_2)$ on Poisson manifolds, we say that the systems are {\bf equivalent} if there exists a Poisson diffeomorphism $\psi$ and a map $\varphi:  \R^s \to  \R^s$ such that we have a commuting diagram:
\begin{align*}
  \begin{diagram}
    \node{(M_1, \Pi_1)}\arrow{e,t}{\psi}\arrow{se,b}{F_1}\node{(M_2, \Pi_2)}\arrow{s,r}{\varphi \circ  F_2}\\
    \node[2]{\R^{s}}
  \end{diagram}
\end{align*}
\subsection{Models for regular Liouville tori}The action-angle theorem for regular Liouville tori on Poisson manifolds proves the equivalence to the following model.\begin{theorem} Let $F=(f_1,\ldots,f_s)$ be an integrable system on the Poisson manifold $(M,\Pi)$, where $\Pi$ has rank $2r = n-s$. Then semilocally around a regular Liouville torus the system is equivalent to
$$(T^\ast \T^r)_{can} \times (\R^{s-r}, \Pi_0, \id),$$
where $(T^\ast \T^r)_{can}$ stands for the model defined in (\ref{eq:can})  and $\Pi_0$ is the zero Poisson structure on $\R^{s-r}$.
\end{theorem}
\begin{proof}
The proof uses the action-angle coordinate theorem for integrable systems on Poisson manifolds (Theorem \ref{thm:aa_poisson}). Semilocally around a regular Liouville torus we have coordinates $(\theta_1,\dots,\theta_r,p_1,\dots,p_{s})$ such that $\Pi=\sum_{i=1}^r \pp{}{\theta_i}\we\pp{}{p_i}$
  and the coordinates $p_1,\ldots, p_s$  depend only on the integrals. Therefore, there exists a map $\varphi$ such that the following diagram commutes:
  \begin{align*}
  \begin{diagram}
    \node{U}\arrow{e,t}{\psi}\arrow{se,b}{F}\node{(T^\ast \T^r \times \R^{s-r}, \Pi_{can}+\Pi_0) }\arrow{s,r}{\varphi \circ  (\mu_{can},\id)}\\
    \node[2]{{\R^{s}}}
  \end{diagram}
\end{align*}
\end{proof} In view of this result, using the constructions in Theorem \ref{thm:cotlift_intsys}, we can take the product of the cotangent lift of an abelian action to $T^\ast M$ with  $(\R^k,\Pi_0,\id)$ to obtain examples of integrable systems in Poisson manifolds;  the integrals are $(f_1,\ldots,f_n,z_1,\ldots,z_k)$, where $z_1,\ldots,z_k$ are the standard coordinates on $\R^k$.

\subsection{Models for singular Liouville tori}The models described above apply to regular Liouville torus of an integrable system.  Other than the $b$-Poisson case, \emph{is there any model for singular tori on general Poisson manifolds?}

Weinstein's splitting theorem \cite{weinstein} gives a local model for a Poisson structure in a neighbourhood of any point. Such a model consists of a local product of a symplectic manifold $S$ with a Poisson manifold which is transversal to $S$ and whose Poisson structure vanishes at a point. This product model extends to several  objects in the Poisson category. However, this is not the case of integrable systems:
As it is proved in \cite{Laurent-Gengoux2010}, an integrable system on a Poisson manifold does not always split locally as a product of an integrable system on a symplectic manifold and an integrable system on the  transverse Poisson manifold. In \cite{camilleeva} cohomological obstructions are studied for such a splitting to exist.

 However such a splitting is possible in a more general class of integrable systems: In \cite{Laurent-Gengoux2010} an action-angle theorem is proved for non-commutative (or superintegrable) systems. A non-commutative integrable system
on a Poisson manifold is given by  an $s$-tuple of functionally independent functions
  $F=(f_1,\dots,f_s)$  such that a subset of $r$ functions $f_1,\dots,f_r$ is in involution with all the functions $f_1,\dots,f_s$;
  ($r+s =\dim M$) and the Hamiltonian vector fields of the functions $f_1, \dots,f_r$ are linearly independent at some point.

 The action-angle coordinate theorem  proved in  \cite{Laurent-Gengoux2010} (Theorem 1.1) gives a semilocal description of the Poisson structure around a standard Liouville torus of a non-commutative integrable system. It provides a system of coordinates $\theta_1,\dots,\theta_{r}$  (\emph{angle coordinates}), $p_1,\dots,p_r$ (\emph{action coordinates}) and$z_1,\dots,z_{s-r}$ (\emph{transverse coordinates}) such that,
\begin{enumerate}
    \item The Poisson structure can be written as

    $       \Pi=\sum_{i=1}^r \pp{}{\theta_i}\we\pp{}{p_i} + \sum_{k,l=1}^{s-r} \phi_{kl}(z) \pp{}{z_k}\we\pp{}{z_l}.$

    \item The leaves of the surjective submersion $F=(f_1,\dots,f_{s})$ are given by the projection onto the
      second component $\T^r \times B^{s}$, in particular, the functions $f_1,\dots,f_s$ depend on
      $p_1,\dots,p_r,z_1, \dots, z_{s-r}$ only.
  \end{enumerate}

In particular this description takes singularities into account and can be used to give explicit models in a neighbourhood of a Liouville torus lying on the singular locus of a Poisson structure in some particular cases, for instance when the transversal Poisson structure is linear (or linearizable).
Another interesting direction of study is the one explored in   \cite{bolsinov2} where singularities of bi-Hamiltonian systems and algebraic properties of Poisson pencils are related.

\end{document}